\documentclass{amsart}

\usepackage{amscd}
\usepackage{amssymb}
\usepackage{soul}
\usepackage{mathrsfs}
\usepackage{enumitem}
\usepackage{graphicx}

\usepackage[all]{xypic}
\usepackage[dvipsnames] {xcolor}
\usepackage{tikz}

\numberwithin{equation}{section}

\theoremstyle{plain} 
\newtheorem{theorem}[equation]{Theorem}
\newtheorem{lemma}[equation]{Lemma}

\newtheorem{proposition}[equation]{Proposition}
\newtheorem{corollary}[equation]{Corollary}

\newtheorem*{RetractionProposition}{Proposition~\ref{proposition: deformation retraction new}}

\newcommand{\argforcustom}{}
\theoremstyle{newtextthm}
\newtheorem{helperforcustom}[equation]{\argforcustom}
\newtheorem*{helperforcustomstar}{\argforcustom}

\theoremstyle{definition}
\newtheorem{definition}[equation]{Definition}

\newtheorem{remark}[equation]{Remark}
\newtheorem*{remark*}{Remark}
\newtheorem{construction}[equation]{Construction}

\newcommand{\defining}[1]{{\emph{#1}}}
\newcommand{\definedas}{:=}

\newcommand{\integers}{{\mathbb{Z}}}

\def\doCal#1{%
\ifx#1\doAllCalEnd\def\doAllCal{\relax}\else%
 \expandafter\edef\csname#1cal\endcsname{{\noexpand\mathcal #1}}\fi}
\def\doAllCal#1{\doCal#1\doAllCal}
\doAllCal ABCDEFGHIJHKLMNOPQRSTUVWXYZ\doAllCalEnd

\def\doBar#1{%
\ifx#1\doAllBarEnd\def\doAllBar{\relax}\else%
 \expandafter\edef\csname#1bar\endcsname{{\noexpand\overline{#1}}}\fi}
\def\doAllBar#1{\doBar#1\doAllBar}
\doAllBar ABCDEFGHIJHLMNOPQRSTUVWXYZabcdefghijhlmnopqrstuvwxyz\doAllBarEnd
\def\doWiggle#1{%
\ifx#1\doAllWiggleEnd\def\doAllWiggle{\relax}\else%
 \expandafter\edef\csname#1wiggle\endcsname{{\noexpand\tilde{#1}}}\fi}
\def\doAllWiggle#1{\doWiggle#1\doAllWiggle}
\doAllWiggle ABCDEFGHIJHLMNOPQRSTUVWXYZabcdefghijklmnopqrstuvwxyz\doAllWiggleEnd

\def\doBold#1{%
\ifx#1\doAllBoldEnd\def\doAllBold{\relax}\else%
 \expandafter\edef\csname#1bold\endcsname{{\noexpand\bf #1}}\fi}
\def\doAllBold#1{\doBold#1\doAllBold}
\doAllBold ABCDEFGHIJHKLMNOPQRSTUVWXYZabcdefghijklmnopqrstuvwxyz\doAllBoldEnd

\def\doBbb#1{%
\ifx#1\doAllBbbEnd\def\doAllBbb{\relax}\else%
 \expandafter\edef\csname#1bb\endcsname{{\noexpand\mathbb{#1}}}\fi}
\def\doAllBbb#1{\doBbb#1\doAllBbb}
\doAllBbb ABCDEFGHIJHKLMNOPQRSTUVWXYZ\doAllBbbEnd

\newcommand{\largestrut}{\mbox{{\large\strut}}}

\makeatletter
\newcommand{\switchmargin}{
\if@reversemargin
\normalmarginpar
\else
\reversemarginpar
\fi
}
\makeatother

\definecolor{llgray}{RGB}{230,230,230}
\definecolor{lblue}{RGB}{217,249,237}
\newcommand{\highlight}[1]{\ifmmode{\text{\sethlcolor{llgray}\hl{$#1$}}}\else{\sethlcolor{llgray}\hl{#1}}\fi}
\newcommand{\highlighteva}[1]{\ifmmode{\text{\sethlcolor{llgray}\hl{$#1$}}}\else{\sethlcolor{lblue}\hl{#1}}\fi}

\DeclareMathOperator{\Aut}{Aut}

\DeclareMathOperator{\BSU}{BSU}
\DeclareMathOperator{\BSO}{BSO}

\DeclareMathOperator{\Hom}{Hom}

\DeclareMathOperator{\id}{id}

\newcommand{\Id}{\mathrm{Id}}
\DeclareMathOperator{\im}{im}

\DeclareMathOperator{\Obj}{Obj}

\DeclareMathOperator{\Out}{Out}

\newcommand{\pdash}{$p$\kern1.3pt-}
\newcommand{\twodash}{$2$\kern1.3pt-}
\newcommand{\whatever}{\text{\,--\,}}

\newcommand{\Zpinfinity}{\integers/{p}^{\infty}}

\newcommand{\size}[1]{\operatorname{size}(#1)}

\newcommand{\subgroupeq}{\subseteq}

\newcommand{\supergroupeq}{\supseteq}
\newcommand{\Slominska}{S{\l}omi\'{n}ska }

\newcommand{\xlongrightarrow}[1]{\xrightarrow{\,#1\,}}

\newcommand\restr[2]{{
  \left.\kern-\nulldelimiterspace 
  #1 
  \vphantom{\big|} 
  \right|_{#2} 
  }}

\newcommand{\suchthat}[1]{\left|\, #1 \right. }

\newcommand{\realize}[1]{\big|#1\big|}
\newcommand{\pcomplete}[1]{{#1}_{p}^{\wedge}}

\renewcommand{\bar}{\widebar} 

\usepackage{hyperref}
\hypersetup{colorlinks=true,linkcolor=violet,citecolor=purple} 

\usepackage{marginnote}

\usepackage[nolabel]{showlabels} 
\definecolor{llgray}{RGB}{230,230,230}

\renewcommand{\L}{\Lcal}

\newcommand{\ints}{\cap}

\newcommand{\too}[1]{\xrightarrow{#1}}

\newcommand{\widebar}[1]{{\overline{#1}}}
\newcommand{\F}{\mathcal{F}}

\newcommand{\longleftrightarrows}{\xymatrix@1@C=16pt{
\ar@<0.4ex>[r] & \ar@<0.4ex>[l]
}}

\renewcommand{\atop}[1]{{\let\scriptstyle\textstyle\let\scriptscriptstyle\scriptstyle\substack{#1}}}

\newcommand{\Pimage}[1]{P_{#1}}
\newcommand{\corestrict}[1]{{\underline{#1}}}
\newcommand{\Phat}{\widehat{P}}

\newcommand{\phibar}{\bar{\phi}}
\newcommand{\alphawiggle}{\widetilde{\alpha}}
\newcommand{\betawiggle}{\widetilde{\beta}}
\newcommand{\varphiwiggle}{\widetilde{\varphi}}
\newcommand{\psiwiggle}{\widetilde{\psi}}

\newcommand{\downto}{\downarrow}
\newcommand{\downtoNoniso}{\downarrow_{\scriptscriptstyle\not\cong}\!}

\newcommand{\Normalizer}[1]{P\downto N_{#1}P}
\newcommand{\NormalizerPrime}[1]{P\downtoNoniso N_{#1}P}

\newcommand{\undercategory}{undercategory}
\newcommand{\undercategories}{undercategories}

\newcommand{\overcategories}{overcategories}

\newcommand{\extension}{\widetilde{P}}

\newtheorem*{maintheorem}{Theorem~\ref{theorem: reduce linking to centric radical}}

\newcommand{\maintheoremtext}{Let $\Fcal$ be a saturated fusion system over a discrete \pdash toral group~$S$, and let $\Lcal$ be a
centric linking system associated to $\Fcal$. 
Let $\Hcal$ be a collection of $\Fcal$-centric subgroups of $S$ that is closed under $\Fcal$-conjugacy and contains all subgroups of $S$ that are both $\Fcal$-centric and $\Fcal$-radical. 
Let $\Hcal^\bullet=\{P^\bullet\suchthat{P\in\Hcal}\}$ and 
assume that $\Hcal^\bullet\subseteq\Hcal$. 
Let $\Lcal^{\Hcal}\subseteq \Lcal$ denote the
full subcategory of $\Lcal$ whose objects are in~$\Hcal$. 
Then the inclusion of $\Lcal^{\Hcal}$ in $\Lcal$ induces a homotopy equivalence
of nerves $\realize{\Lcal^{\Hcal}}\simeq \realize{\Lcal}$.}

\newcommand{\RetractionPropositionText}{
Let $(S, \Fcal, \Lcal)$ be a \pdash local compact group, and let $P$ be a fully $\Fcal$-normalized subgroup of~$S$.
There is a retraction 
$r\colon(P\downtoNoniso\Lcal)\longrightarrow(\NormalizerPrime{\Lcal})$
that induces a homotopy equivalences of nerves. 
}

\title{Subgroup collections controlling the homotopy type of a \pdash local compact group}

\author{Eva Belmont}
\address{Department of Mathematics, University of California San Diego, La Jolla, CA, USA}
\email{ebelmont@ucsd.edu}

\author{Nat\`alia Castellana}
\address{Departament de Matem\`atiques, Universitat Aut\`onoma de Barcelona, and Centre de Recerca Matemàtica, Barcelona, Spain}
\email{natalia@mat.uab.cat}

\author{Kathryn Lesh}
\address{Department of Mathematics, Union College, Schenectady NY, USA}
\email{leshk@union.edu}



\begin{document}
\maketitle
\markboth{\sc{Belmont, Castellana, and Lesh}}{\sc{Subgroups of \pdash local compact groups}}

\begin{abstract}
    Let $(S,\Fcal,\Lcal)$ be a \pdash local compact group. We prove that the (uncompleted) homotopy type of the nerve of the linking system $\Lcal$ is determined by the collection of subgroups of $S$ that are $\Fcal$-centric and $\Fcal$-radical. This result generalizes the result for the case of \pdash local finite groups, which is in the literature. 
\end{abstract}

\section{Introduction}
The structure of a ``\pdash local compact group" was introduced by Broto, Levi, and Oliver in \cite{BLO-Discrete}
and provides a common framework for the study of the mod~$p$ homotopy type of various types of classifying spaces. 
Examples include classical objects such as \pdash completed classifying spaces of finite groups and compact Lie groups. More broadly, one can use the framework to study classifying spaces of homotopy-theoretic generalizations of groups, 
such as 
\pdash compact groups \cite{dwyer-wilkerson-fixed-point}. For example,
the \pdash completed classifying space of a finite loop space is the classifying space of \pdash local compact group \cite{BLO-LoopSpaces}. Other examples constructed from exotic \pdash local finite groups are described in \cite{GLR}.


To describe a \pdash local compact group, one begins with a ``discrete \pdash toral" group~$S$ (Definition~\ref{def:dpt}). A ``fusion system" over~$S$
(Definition~\ref{definition: Fusion})
is a subcategory of the category of groups, with objects given by all subgroups of~$S$. 
``Saturated" fusion systems (Definition~\ref{definition: SaturatedFusion}) satisfy additional axioms requiring $S$ to behave like a Sylow \pdash subgroup of the hypothetical supergroup~$G$, and require the morphism sets of the fusion system to behave as though they were homomorphisms induced by conjugation in~$G$ (even 
though such a supergroup may not exist). 

Associated to a fusion system $\Fcal$ is a ``centric linking system"~$\Lcal$ (Definition~\ref{definition: linking}).
While morphism sets in $\Fcal$ mimic \emph{homomorphisms} induced by conjugation, an associated linking system $\Lcal$ for $\Fcal$ has morphism sets that mimic \emph{group elements}
of a hypothetical supergroup~$G$ that induce the homomorphisms in~$\Fcal$ via conjugation. 
The nerve $\realize{\Lcal}$ is analogous to the classifying space of a group, $BG$. 
And indeed, if a \pdash local compact group arises from a compact Lie group, then the \pdash completion of the nerve of $\Lcal$ is a model for the \pdash completion of~$BG$ \cite[Thm.~9.10]{BLO-Discrete}, and the same is true if $\Lcal$ arises from a \pdash compact group \cite[Thm.~10.7]{BLO-Discrete}.

Classifying spaces of compact Lie groups and finite groups admit mod~$p$ homology decompositions in terms of orbit categories with respect to certain families of subgroups.
Such decompositions are a key tool in the study of homotopy uniqueness and of maps between classifying spaces \cite{Dwyer-Homology, JMO,JM}. 
Similar decompositions exist for \pdash local compact groups \cite[Prop.~4.6]{BLO-Discrete}. 
The two taken together are key ingredients in the proofs that the \pdash local compact groups associated to Lie groups and finite groups model the \pdash completion of the groups' classifying spaces.

Centric and radical subgroups of a group $G$ play a key part in the homotopy type of~$BG$. For example, Dwyer~\cite{Dwyer-Homology} showed that the collection of \pdash radical, \pdash centric subgroups
of a finite group $G$ are enough to recover the \pdash completed homotopy type of $BG$ for any of the classical homology decompositions, and later the same was established for compact Lie groups (\cite{JMO}, \cite{Libman-Minami}).

In fusion systems over \emph{finite} \pdash groups, the result that centric and radical subgroups determine the homotopy type of the classifying space was proved in \cite[Thm.~3.5]{BCGLO}, using an induction via ``pruning subgroups." The corresponding result for \pdash local compact groups is not in the literature, and 
that is the gap that we fill with this paper. 

We make use of the ``bullet construction" of \cite{BLO-Discrete}:  $\Lcal^{\bullet}\subseteq \Lcal$ is a  full subcategory of the centric linking system  such that the inclusion $\Lcal^\bullet\subseteq\Lcal$ induces an homotopy equivalence on nerves
\cite[Prop.~4.5]{BLO-Discrete}. Attractively for computation, $\Obj(\Lcal^\bullet)$ contains finitely many $S$-conjugacy classes of subgroups (\cite[Lemma~3.2]{BLO-Discrete}) and contains all subgroups that are both centric and radical. 

\begin{theorem}   \label{theorem: reduce linking to centric radical}
\maintheoremtext
\end{theorem}

Our proof follows the same general strategy as \cite{BCGLO} but we clarify and streamline their argument, and handle some extra obstacles that occur because the linking system is not finite. Theorem~\ref{theorem: reduce linking to centric radical} is also closely related to results in Appendix~A of \cite{BLO-LoopSpaces}, in which 
the authors study the mod $p$ homotopy type of 
transporter systems, another type of category 
used to describe the classifying space of a \pdash local compact group. Corollary~A.10 of \cite{BLO-LoopSpaces} is similar to Theorem~\ref{theorem: reduce linking to centric radical}, but assumes that the collection of subgroups being considered is closed under supergroups, which 
is not the case when dealing with $\Fcal$-centric $\F$-radical subgroups. 

Another advantage of our approach is that Theorem~\ref{theorem: reduce linking to centric radical} gives 
a genuine homotopy equivalence of nerves, whereas the techniques of \cite[Cor.~A.10]{BLO-LoopSpaces} necessarily can only give equivalences after \pdash completion because there is a mod~$p$ homology argument involved.
It is true that results such as \cite[Prop.~1.1]{BLO-FiniteGroups} for the \pdash local finite group $(S, \Fcal_S(G), \Lcal_S(G))$ associated to a finite group~$G$, along with the similar result \cite[Thm~9.10]{BLO-Discrete} for compact Lie groups, only tell us that $|\Lcal_S(G)|$ agrees with $BG$ after \pdash completing both sides.
However, more precise versions of this statement have been obtained in some cases without \pdash completion. 
For example, Libman and Viruel~\cite{Libman-Viruel} give conditions on a \pdash local finite group $(S,\Fcal,\Lcal)$ such that 
$\realize{\Lcal} \simeq B\Gamma$ for some discrete group~$\Gamma$. See also \cite{COS} for a different example identifying the uncompleted nerve, this time in the simply-connected case. 

Work of \Slominska \cite{Slominska-hocolim} allows one to describe the homotopy type of $\realize{\Lcal}$ as a homotopy colimit indexed on a poset. This approach was taken by Libman \cite{Libman-normalizer} to describe a ``normalizer decomposition'' for \pdash local finite groups, 
which in particular gives a decomposition of the uncompleted nerve of the linking system. In forthcoming work \cite{WIT-main}, the current authors, together with Grbi\'{c} and Strumila, prove an analogous theorem for \pdash local compact groups. 
In this context, Theorem~\ref{theorem: reduce linking to centric radical} 
reduces the size of the indexing category for the decomposition and allows for explicit computations, in some cases giving homotopy pushout descriptions for~$\realize{\Lcal}$. For example, the general normalizer decomposition recovers the homotopy pushout descriptions for $\BSU(2)$ and $\BSO(3)$ originally due to Dwyer, Miller, and Wilkerson \cite{DMW1}.



\subsection*{Organization.}
Section~\ref{section: background} gathers background material on \pdash local compact groups. Section~\ref{section: normalizer fusion subsystems} discusses normalizer fusion and linking systems, adapting results from \cite{BCGLO}. In Section~\ref{section: Phat}, for an arbitrary fully normalized $P$ in~$\Fcal$, we construct~$\extension$, the largest supergroup of $P$ over which all $\Fcal$-automorphisms of $P$ extend, and we show that $\extension$ coincides with the group $\Phat$ used in \cite{BCGLO}. In Section~\ref{section: proof of theorem} we prove the main theorem.

\subsection*{Acknowledgements.}
The first author was supported by NSF grant  DMS-2204357 and by an AWM-NSF mentoring travel grant.
The second author was partially supported by Spanish State Research Agency project PID2020-116481GB-I00, 
the Severo Ochoa and María de Maeztu Program for Centers and Units of Excellence in R$\&$D (CEX2020-001084-M), and the CERCA Programme/Generalitat de Catalunya. All three authors acknowledge the hospitality of the program ``Higher algebraic structures in algebra, topology and geometry" at the Mittag-Leffler Institute in Spring~2022. 

\bigskip

\section{Background}
\label{section: background}

In this section, we gather definitions and preparatory results. We review the definition of a \pdash local compact group, and we also review some lemmas on linking systems. 

\begin{definition}\label{def:dpt}
A \defining{discrete \pdash toral group} is a group $P$ given by an extension
\[
1\longrightarrow \left(\Zpinfinity\right)^k \longrightarrow P\longrightarrow \pi_{0}P\longrightarrow 1,
\]
where $k$ is a nonnegative integer which we call the \defining{rank}, and $\pi_{0}P$ is a finite \pdash group. 
\begin{itemize}
    \item 
The group $\left(\Zpinfinity\right)^k$ is the \defining{identity component of~$P$}, denoted~$P_0$. 
\item 
The \defining{size} of a discrete \pdash toral group $P$ is an ordered pair $\size{P}\definedas(k, c)$, where $k$ is the rank of $P$ and $c=|\pi_0 P|$. The pairs
$(k,c)$ are given the lexicographic order (see~\cite[A.5]{CLN}).
\end{itemize}
\end{definition}

Subgroup inclusions respect size as in the finite group case: 
if $P_1\subgroupeq P_2$ are discrete \pdash toral groups, then $\size{P_1}\leq \size{P_2}$, with equality if only if $P_1 = P_2$
(see~\cite[Sec.~1]{BLO-Discrete}).

\begin{definition}\cite[Defn.~2.1]{BLO-Discrete}
\label{definition: Fusion}
A \defining{fusion system} $\F$ over a discrete \pdash toral group $S$ is a
subcategory of the category of groups, defined as follows. The objects of
$\Fcal$ are all of the subgroups of~$S$. The morphism sets $\Hom_{\F}(P,Q)$
contain only group monomorphisms, and satisfy the following conditions.
\begin{enumerate}[label=(\alph*)]
   \item $ \Hom_S(P,Q) \subseteq \Hom_{\F}(P,Q)$ for all $ P,Q\subgroupeq S $. In
particular, all subgroup inclusions and conjugations by elements of $S$ are in~$\Fcal$.
   \item Every morphism in $\Fcal$ factors as the composite of an isomorphism in $\Fcal$ followed by a subgroup inclusion.
\end{enumerate}
\end{definition}

We think of the homomorphisms of a fusion system $\Fcal$ as mimicking the idea of conjugation in a supergroup $G$ of~$S$. Accordingly, two groups $P$ and $P'$ that are objects of $\Fcal$ are called \defining{$\Fcal$-conjugate} if they are isomorphic as objects of~$\Fcal$. 

\begin{definition}   
\label{definition: fully}
Let $\Fcal$ be a fusion system over the discrete \pdash toral subgroup~$S$. 
\begin{enumerate}
    \item We say that $P\subgroupeq S$ is \defining{fully centralized} if for all $Q\subgroupeq S$ that are $\Fcal$-conjugate to~$P$, we have 
    $\size{C_S P} \geq \size{C_S Q}$.
    \item We say that $P\subgroupeq S$ is \defining{fully normalized} if for all $Q\subgroupeq S$ that are $\Fcal$-conjugate to~$P$, we have $\size{N_S P}\geq \size{N_S Q}$.
\end{enumerate}
\end{definition}

The following definition, of ``saturation," is intended to axiomatize the consequences of the group $S$ being a Sylow \pdash subgroup of~$G$, together with morphisms coming from conjugation by elements of~$G$. This is a technical condition that is assumed in order to guarantee good group-like properties, as in Definition~\ref{definition: classifying space of linking}.

\begin{definition}\cite[Defn.~2.2]{BLO-Discrete}
\label{definition: SaturatedFusion}
A fusion system $\Fcal$ is \defining{saturated} if the following three conditions hold:
\begin{enumerate}[label=(\Roman*)]
\item If $P\subgroupeq S$ is fully normalized in $\Fcal$, 
then $P$ is fully centralized in $\Fcal$, 
the group $\Out_\Fcal(P)\definedas\Aut_\Fcal(P)/\Aut_P(P)$ is finite, 
and the group $\Out_S(P)\definedas\Aut_S(P)/\Aut_P(P)$ is a Sylow \pdash subgroup of $\Out_\Fcal(P)$.
\item If $P\subgroupeq S$ and $\phi\in \Hom_\Fcal(P,S)$ 
are such that $\phi(P)$ is fully centralized, and if we set
\[
N_\phi \definedas 
\left\{ g\in N_S \suchthat{\phi c_g \phi^{-1}\in \Aut_S(\phi(P))} \right\}, 
\]
then there exists $\phibar\in \Hom_\Fcal(N_\phi,S)$ such that
$\restr{\phibar}{P}=\phi$.
\item If $P_1\subgroupeq P_2\subgroupeq P_3\subgroupeq \dots$ 
is an increasing sequence of subgroups of~$S$ with the property that
$P_\infty = \bigcup_{n=1}^\infty P_n$, and if
$\phi\in \Hom(P_\infty,S)$ is any homomorphism such that $\restr{\phi}{P_n}\in\Hom_\Fcal(P_n,S)$ for all~$n$, 
then $\phi\in \Hom_\Fcal(P_\infty,S)$.
\end{enumerate}
\end{definition}

The goal of this paper is to show one can safely restrict to a sub-collection of subgroups of~$S$; the objects we consider are exactly those that satisfy both of the following two conditions.
\begin{definition}\label{definition:Fcentric-Fradical}
Let $\Fcal$ be a fusion system over a discrete \pdash toral group~$S$.
\begin{enumerate}
\item \label{item: discrete Fcentric}
     A~subgroup $P\subgroupeq S$ is called \defining{$\Fcal$-centric}
     if $P$ contains all elements of $S$ that centralize it,
     and likewise all $\Fcal$-conjugates of $P$ contain their
     $S$-centralizers.
\item \label{item: discrete Fradical}
     A subgroup $P\subgroupeq S$ is called \defining{$\Fcal$-radical}
     if $\Out_{\Fcal}(P)$ 
     contains no nontrivial normal \pdash subgroup.
\end{enumerate}
\end{definition}

A linking system, whose definition we recall next, has more morphisms than the fusion system. The motivating example satisfies $\pcomplete{\realize{\Lcal}} \simeq \pcomplete{\realize{BG}}$ for a compact Lie group $G$, though linking systems are more general than this.

\begin{definition}\cite[Defn.~4.1]{BLO-Discrete}
\label{definition: linking}
Let $\Fcal $ be a fusion system over a discrete \pdash toral group~$S$.
A \defining{centric linking system associated to~$\Fcal$} is a category $\L$ whose
objects are the $\Fcal$-centric subgroups of $S$, together with a functor
$
\L \xlongrightarrow{\pi} \Fcal
$
and ``distinguished'' monomorphisms $P\xrightarrow{\delta_P} \Aut_\Lcal(P)$ for each
$\Fcal$-centric subgroup $P\subgroupeq S$ satisfying the following
conditions.
\begin{enumerate}[label=(\Alph*)]
\item $\pi$ is the identity on objects and surjective on morphisms. More
precisely, for each pair of objects $P,Q\in \Lcal$, the center $Z(P)$ acts freely on
$\Hom_\Lcal(P,Q)$ by composition (upon identifying $Z(P)$ with
$\delta_P(Z(P))\subgroupeq \Aut_\Lcal(P)$), and $\pi$ induces a bijection
\[
\Hom_{\L}(P,Q)/Z(P)\xrightarrow{\ \cong\ } \Hom _{\Fcal}(P,Q).
\]
\item For each $\Fcal$-centric subgroup $P\subgroupeq S$ and each $g\in P$, the functor
$\pi$ sends $\delta_P(g)\in\Aut_\Lcal(P)$ to $c_g\in \Aut_\Fcal(P)$.
\item For each $ f \in \Hom _{\Lcal}(P,Q) $ and each $g \in P$, the following
square commutes in $\Lcal$:
\[
\diagram
P \rto^{f} \dto_{\delta _P(g)}
      & Q \dto^{\delta_Q(\pi(f)(g))}
\cr P \rto^f & Q.
\enddiagram
\]
\end{enumerate}
\end{definition}
 
With these definitions in place, we arrive at the object of study. 

\begin{definition}     \label{definition: classifying space of linking}
A \defining{\pdash local compact group} is a triple $(S,\Fcal,\Lcal)$, where $\Fcal$ is
a saturated fusion system over the discrete \pdash toral group~$S$,
and $\Lcal$ is a centric linking system associated to~$\Fcal$. The
\defining{classifying space of $(S,\Fcal,\L)$} is defined as $B\Fcal\definedas \pcomplete{\realize{\Lcal}}$.
\end{definition}

A priori, computing the classifying space of a \pdash local compact group requires handling an infinite number of isomorphism classes of objects of~$\Lcal$. However, 
Broto, Levi, and Oliver constructed a functorial retraction 
$(\whatever)^\bullet\colon\Fcal\rightarrow\Fcal$ that lifts to the associated linking system, and whose image contains a finite number of conjugacy classes.

\begin{proposition}
\cite[Defn.~3.1, Lemma~3.2, Prop.~3.3, Prop.~4.5]{BLO-Discrete}
\label{proposition: bullet}
Let $\Fcal$ be a saturated fusion system over a discrete \pdash toral group~$S$.
There is an idempotent endofunctor $(\whatever)^\bullet\colon \Fcal \longrightarrow\Fcal$, the \defining{bullet functor}, 
such that the full subcategory $\Fcal^{\bullet}\subseteq\Fcal$ with 
$\Obj(\Fcal^{\bullet})\definedas \left\{P^\bullet \suchthat{P\subgroupeq  S}\right\}$
is closed under $\Fcal$-conjugacy, contains finitely many
$S$-conjugacy classes, and contains all subgroups $P\subseteq S$ that are both
$\Fcal$-centric and $\Fcal$-radical. If $\Lcal$ is a linking system associated to~$\Fcal$, then $(\whatever)^\bullet$ lifts to an idempotent endofunctor of~$\Lcal$, 
and the inclusion $\Lcal^\bullet\subseteq\Lcal$ induces a homotopy equivalence of nerves. 
\end{proposition}

Proposition~\ref{proposition: bullet} says that, for computational purposes, we can restrict to the category $\Lcal^\bullet$, which has a finite number of isomorphism classes of objects. 
Our goal in this paper is to show that one can restrict to a yet smaller collection of objects, namely those that are both $\Fcal$-centric and $\Fcal$-radical, without changing the nerve of the associated linking system.

\smallskip

The remainder of this section gathers lemmas related to lifting morphisms from a fusion system to the associated linking system.
In a fusion system, all of the morphisms between subgroups are actual group homomorphisms, but morphisms in a linking system cannot be viewed in this way.
Given a morphism $\varphi$ in~$\Lcal$, there is an \textit{associated} homomorphism of groups, namely the homomorphism $\pi(\varphi)$ in $\Fcal$. But the morphisms in $\Fcal$ are analogous to group homomorphisms induced by conjugation in a supergroup $G\supergroupeq S$, while the morphisms in $\Lcal$ itself are analogous to the group elements that induce the homomorphism. 

Nevertheless, the last three lemmas of this section establish that several common features of group homomorphisms also exist for morphisms in~$\Lcal$.
First, we can uniquely complete liftings from $\Fcal$ to~$\Lcal$.

\begin{lemma}\cite[Lemma~4.3]{BLO-Discrete}
\label{lemma: factoring lemma}
Given morphisms $\varphi\in \Hom_\Fcal(P,Q)$ and $\psi\in\Hom_\Fcal(Q,R)$, and lifts $\psiwiggle$ and $\widetilde{\psi\varphi}$ of $\psi$ and $\psi\circ\varphi$, respectively, to~$\Lcal$, there is a unique compatible lift of~$\varphi$ to $\Lcal$ making the diagram on the right a commuting lift to $\Lcal$ of the diagram in $\Fcal$ on the left.  
\[
\xymatrix{
P \ar[r]^{\varphi} \ar[d]_-{\id}
   &Q\ar[d]^-{\psi}\\
P \ar[r]_-{\psi\circ\varphi} &R 
}
\hspace{60pt}
\xymatrix{
P \ar@{-->}[r]^{\exists ! \varphiwiggle} \ar[d]_-{\id}
   &Q\ar[d]^-{\psiwiggle}\\
P \ar[r]_-{\widetilde{\psi\varphi}} &R 
}
\]
\end{lemma}

Next we need an analogue in $\Lcal$ of inclusions. Given $P\subgroupeq Q$, there is a preferred morphism $i:P\to Q$ in $\Fcal$, namely the subset inclusion. In $\Lcal$ there is no natural notion of subgroup inclusion, but the next lemma says that we can make a coherent choice of lifts of the subgroup inclusions in $\Fcal$ to morphisms in $\Lcal$.

\begin{lemma}
\label{lemma: distinguished inclusions}
\cite[Prop.~1.5, Rem.~1.6]{JLL}
The poset of inclusions of subgroups in $\Fcal$ lifts to a compatible sub-poset $\big\{\iota_P^Q:P\to Q\big\}$ of~$\Lcal$. In particular, $\iota_P^P$ is the identity morphism of~$P$, and given inclusions $P\subgroupeq Q \subgroupeq R$ we have $\iota_Q^R\circ \iota_P^Q = \iota_P^R$.
\end{lemma}

Lastly, we need restriction and corestriction of morphisms in fusion and linking systems. 
In a fusion system, a morphism 
$\varphi\colon P\to Q$ can be restricted to a subgroup
$A\subseteq P$ because the subgroup inclusion $A\hookrightarrow P$ is necessarily a morphism of~$\Fcal$ (Definition~\ref{definition: Fusion}(a)). Similarly, if $P_\varphi\definedas\im(\varphi\colon P\rightarrow Q)$, then the isomorphism $P\xrightarrow{\varphi} P_\varphi$ is a morphism of~$\Fcal$  
(Definition~\ref{definition: Fusion}(b)),
and we call it the corestriction of~$\varphi$ to~$P_\varphi$.

Once we have fixed a compatible subposet of inclusions in $\Lcal$ as in Lemma~\ref{lemma: distinguished inclusions}, we can also define restrictions and corestrictions in the linking system. 
For $A\subgroupeq P$, a morphism $\varphi\in
\Hom_\Lcal(P,Q)$ has a restriction $\restr{\varphi}{A} \definedas \varphi\circ
\iota_{A}^P$ in $\Hom_\Lcal(A,Q)$. The first part of the next lemma says that there is also a
unique corestriction $\corestrict{\varphi}$ to 
$P_\varphi\definedas\im(\pi(\varphi)\colon P\rightarrow Q)$. The same reference that we cite shows that morphisms in $\Lcal$ can be corestricted to any subgroup
containing~$P_\varphi$, but we do not need this generality.

\begin{lemma}\label{lemma: corestriction}
~\begin{enumerate}
\item \cite[Lemma 1.7(i)]{JLL} Let $\varphi\in \Hom_\Lcal(P,Q)$ and let $P_\varphi:= \pi(\varphi)(P)$.
There is a unique map $\corestrict{\varphi}\in \Hom_\Lcal(P,P_\varphi)$ such that $\varphi =
\iota_{P_\varphi}^Q \circ \corestrict{\varphi}$, giving a commuting lift to $\Lcal$ (on the right) of the commuting diagram in $\Fcal$ (on the left).
\[
\xymatrix{
P \ar[r]^-{\strut\corestrict{\pi(\varphi)}} 
  \ar[d]_-{\id}
     & P_\varphi\ar[d]^{\subgroupeq}
\\
P\ar[r]_{\pi(\varphi)}& Q
}
\hspace{60pt}
\xymatrix{
P \ar[r]^{\corestrict{\varphi}} 
  \ar[d]_-{\id}
    & P_\varphi\ar[d]^{\iota_{P_\varphi}^Q}
\\
P\ar[r]_{\varphi} & Q
}
\]

\item \cite[Prop.~A.2, Cor.~A.5]{BLO-LoopSpaces} Given an isomorphism $\varphi:P\to P'$ in $\Fcal$, every lift of $\varphi$ to $\Lcal$ is
an isomorphism.
\item Given a diagram in $\Fcal$ on the left, and a lift $\varphiwiggle$ of $\varphi$ to~$\Lcal$, there is a unique lift $\restr{\varphiwiggle}{P}$ of $\restr{\varphi}{P}$ making the diagram on the right commute in~$\Lcal$, and if $\restr{\varphi}{P}$ is an isomorphism, so is $\restr{\varphiwiggle}{P}$.
\[
\xymatrix{
P\ar[r]^-{\restr{\varphi}{P}}
 \ar[d]_-{\subgroupeq}
& Q
 \ar[d]^-{\subgroupeq}
\\
P'\ar[r]^-\varphi\ar[r] & Q'
}
\hspace{60pt}
\xymatrix{
P   \ar[r]^-{\restr{\varphiwiggle}{P}}
    \ar[d]_-\iota 
& Q\ar[d]^-\iota
\\
P'\ar[r]^-\varphiwiggle & Q'
}
\]
\end{enumerate}
\end{lemma}

\begin{proof}
For the first statement in~(3), apply (1) to $\varphiwiggle\,\circ\, \iota_P^{P'}$ (see also \cite[Prop.~2.11]{Libman-normalizer}). The second statement in (3) follows from~(2).
\end{proof}


\section{Normalizer fusion subsystems}
\label{section: normalizer fusion subsystems}

Quillen's Theorem~A allows the establishment of a homotopy equivalence between the nerve of a category and the nerve of a subcategory by studying the nerves of \overcategories\ or \undercategories\  for the inclusion. The proof of Theorem~\ref{theorem: reduce linking to centric radical} relies on \undercategories\ for its inductive strategy. In this section, we establish the first of a sequence of equivalences necessary for the proof. 

Suppose that $(S,\Fcal,\Lcal)$ is a \pdash local compact group, and $P\subgroupeq S$. Let $P\downto\Lcal$ be the \undercategory\ of~$P$: objects are morphisms $P\rightarrow Q$ in~$\Lcal$, and morphisms are commuting triangles.
Let $P\downtoNoniso\Lcal$ denote the full subcategory  of $P\downto\Lcal$ consisting of objects $P\rightarrow Q$ that are non-isomorphisms of~$\Lcal$. 
To analyze $P\downtoNoniso\Lcal$, whose nerve will be one stage in a sequence of equivalences, we use the 
``normalizer fusion subsystem" $\Normalizer{\Lcal}$ described in 
\cite[Sec.~2]{BLO-LoopSpaces} for \pdash local compact groups. We follow \cite{BCGLO}, where the analogue of Theorem~\ref{theorem: reduce linking to centric radical} is established for \pdash local \emph{finite} groups. The goal of the section is to prove the following. 

\begin{proposition}    \label{proposition: deformation retraction new}
\RetractionPropositionText
\end{proposition}

We give the definition of a normalizer subsystem and basic lemmas, and then follow \cite{BCGLO} in defining a retraction functor. The proof of Proposition~\ref{proposition: deformation retraction new} concludes the section. Most of the section consists of suitable specialization or generalization of results of \cite{BLO-LoopSpaces} and \cite{BCGLO}. 

\begin{definition}    \label{definition: normalizer fusion new}
Let $\Fcal$ be a saturated fusion system over a discrete \pdash toral group~$S$, and let $P\subgroupeq S$ a fully $\Fcal$-normalized subgroup. The 
\defining{normalizer of~$P$ in~$\Fcal$}, denoted
$N_{\Fcal}P$, is a fusion system over the discrete $p$-toral group $N_S P$. If 
$Q,R\subgroupeq N_S P$, then $\Hom_{N_{\Fcal}P}(Q,R)$ is given by
\[
\Big\{ \varphi\in \Hom_{\Fcal}(Q,R) 
\suchthat{\largestrut
\exists~\varphi'\in \Hom_{\Fcal}(Q\cdot P,R\cdot P)
\mbox{\,with\,}\restr{\varphi'}{Q}=\varphi
\mbox{ and }\varphi'(P)=P}
\Big\}.
\]
\end{definition}

Note that the objects $Q$ and $R$ in Definition~\ref{definition: normalizer fusion new} are subgroups of~$N_S P$, but do not have to contain $P$ itself. If they happen to contain~$P$, then $\Hom_{N_\Fcal P}(Q,R)$ is just the subset of $\Hom_{\Fcal}(Q,R)$ of consisting of morphisms that take $P$ to~$P$. Otherwise, the definition is requiring that morphisms extend to the smallest subgroups that \emph{do} contain~$P$, namely $Q\cdot P$ and $R\cdot P$, in a way that takes $P$ to~$P$. 

Saturation is a key technical requirement, and fortunately is inherited by the normalizer fusion system. 

\begin{lemma}
\cite[Thm.~2.3]{BLO-LoopSpaces}
If $\Fcal$ is a saturated fusion system, then so is~$N_\Fcal P$.
\end{lemma}

We would like an associated centric linking system. If $\Lcal$ is a centric linking system associated to~$\Fcal$, there is a candidate linking system $N_\Lcal P$ associated to $N_{\Fcal} P$ that is given by a subcategory of~$\Lcal$. 
In the following definition, note that subgroups of $N_S P$ that are $\Fcal$-centric are necessarily $N_\Fcal P$-centric as well.

\begin{definition}  \label{definition: linking normalizer new}

The category $N_\Lcal P$, the \defining{normalizer in $\Lcal$ of~$P$}, 
is defined as a subcategory of~$\Lcal$. 
The object set of $N_\Lcal P$ is given by $N_{\Fcal}P$-centric subgroups.
The morphism sets $\Hom_{N_\Lcal P}(Q,R)$ are given by
\[
\left\{ 
\varphi\in \Hom_{\Lcal}(Q\cdot P, R\cdot P) 
\suchthat{\restr{\pi(\varphi)}{Q}\in \Hom_{N_{\Fcal} P}(Q,R) \mbox{ and }\pi(\varphi)(P)=P}
\right\},
\]
where $\pi\colon \Lcal \to \Fcal$ is the projection from the linking system to the fusion system. 
\end{definition}

\begin{lemma}\cite[Lemma~1.21]{Gonzalez}
If $P$ is a fully normalized subgroup in~$\Fcal$, 
then the category $N_\Lcal P$ is a centric linking system associated to $N_{\Fcal} P$. 
\end{lemma}

Before we go on to \undercategories, we pause to note easy properties of $N_\Fcal P$.

\begin{lemma}   \label{lemma: P in Q}
If $\varphi\colon P\rightarrow Q$ is a morphism in $N_\Fcal P$, then $P\triangleleft Q$ and $\varphi(P)=P$.
\end{lemma}

\begin{proof}
If $\varphi\colon P\rightarrow Q$ is a morphism in $N_\Fcal P$, then the definition says that there is a morphism $\varphi'\colon P\rightarrow Q\cdot P$ such that $\restr{\varphi'}{P}=\varphi$ and $\varphi'(P)=P$. Hence $\varphi(P)=P$. Since $Q\subgroupeq N_S P$, we know $P$ is normal in~$Q$. 
\end{proof}

Next we define the categories used for the proof of the main result.

\begin{definition}  \label{definition: linking normalizer new categories}
\hfill
\begin{enumerate}

\item We write $\Normalizer{\Lcal}$ to denote the 
\undercategory\   (or ``coslice category") of $P$ in $N_\Lcal(P)$, i.e. the category whose objects are morphisms $P\rightarrow Q$ of $N_\Lcal P$ and whose morphisms are commuting triangles under~$P$. 
\item We write $\NormalizerPrime{\Lcal}$ (resp., $P\downtoNoniso\Lcal$) to denote the full subcategory of $\Normalizer{\Lcal}$ (resp. $P\downto\Lcal$) whose objects are maps $P\rightarrow Q$ that are not isomorphisms in~$\Lcal$. 
\end{enumerate}
\end{definition}

The bulk of the work of this section is to construct a retraction functor $r$ from $P\downarrow \Lcal$ to $\Normalizer{\Lcal}$
and show that the retraction restricts to a functor from 
$P\downtoNoniso\Lcal$ to $\NormalizerPrime{\Lcal}$ (following the model of \cite{BCGLO}). 
Let $P\xlongrightarrow{\varphi} Q$ be an object in $P\downto\Lcal$. In an ideal world, we would like to construct an object $P\rightarrow Q_\varphi$ of $\Normalizer{\Lcal}$
that depends only on~$\varphi$. Sometimes this works: 
if $\varphi(P)=P$, then we will indeed be able to take $Q_\varphi = Q \ints
N_SP$. In general, however, we only have 
$\varphi(P)\cong P$; we would like to
twist everything by this isomorphism, but it turns out that such a twist requires making a choice, and unfortunately, the group
$Q_\varphi$ is not uniquely defined by~$\varphi$.
However there will still be coherent maps between the outcomes for all of the objects $\varphi$ in $P\downarrow\Lcal$, making the retraction into a functor.

\begin{construction}   \label{construction: retraction to normalizer new}
For a fully normalized subgroup $P$ and a morphism $P\xlongrightarrow{\varphi} Q$ in~$\Lcal$, let  
$\Pimage{\varphi}
    \definedas \im(\pi(\varphi)\colon P\rightarrow Q)$. 
By Lemma~\ref{lemma: corestriction}, there exists a unique morphism ${\corestrict{\varphi}}$ in $\Lcal$ such that $\varphi=\iota_{\Pimage{\varphi}}^{Q}\circ{\corestrict{\varphi}}$ (the ``corestriction" of $\varphi$ to~$P_\varphi$). 
Then $\Pimage{\varphi}$ is necessarily $\Fcal$-isomorphic to $P$, so there exists a morphism 
$f_{\varphi}\colon N_S(\Pimage{\varphi})\rightarrow N_S P$ in $\Fcal$ whose restriction to $\Pimage{\varphi}$
corestricts to an isomorphism to $P\subseteq N_S P$ (see \cite[Lemma~2.2(b)]{BLO-LoopSpaces} with $K=\Aut P$). If $\varphi$ happens to be a morphism in $N_\Lcal P$, i.e. if $P_\varphi=P$, then we choose $f_\varphi$ to be the identity map. Then define
\[
Q_{f_\varphi}\definedas
\im\left[\largestrut
f_{\varphi}\colon
  N_S(\Pimage{\varphi})\cap Q 
  \rightarrow N_S P
\right].
\]
In~$\Fcal$ we have a diagram with restriction/corestriction morphisms of $f_\varphi$ labelled with underlines:
\begin{equation}   \label{diagram: factoring diagram in Fcal new}
\begin{gathered}
\xymatrix{
P  
  \ar[r]^{\cong}
  \ar[ddr]_-{\pi({\corestrict{\varphi}})}
  &P \ar[r]^-{\subseteq}
  & Q_{f_\varphi} 
     \ar[rr]^-{\subseteq}
     \ar[rd]^-{\subgroupeq\circ({\corestrict{f_\varphi}})^{-1}}
  && N_{S}(P)
\\
&&&Q\\
&\Pimage{\varphi}  
  \ar[r]_-{\subseteq}
  \ar[uu]_-{\corestrict{\corestrict{f_\varphi}}}^{\cong}
  &N_S(\Pimage{\varphi})\cap Q 
     \ar[rr]_-{\subseteq}
     \ar[uu]_-{{\corestrict{f_\varphi}}}^{\cong}
     \ar[ur]_-{\subseteq}
&& N_{S}(\Pimage{\varphi}).
  \ar[uu]_-{f_\varphi}
}
\end{gathered}
\end{equation}

\end{construction}

We want to lift diagram \eqref{diagram: factoring diagram in Fcal new} 
to~$\Lcal$, as shown in \eqref{diagram: factoring diagram in Lcal} below. We first choose a lift $F_\varphi$ for~$f_\varphi$ (using the identity for $F_\varphi$ if $P=\Pimage{\varphi}$),
and then we lift the subset containments of 
\eqref{diagram: factoring diagram in Fcal new} as the preferred inclusion maps of~$\Lcal$, all marked simply as $\iota$ to declutter the diagram. By Lemma~\ref{lemma: corestriction}, there are unique
isomorphisms ${\corestrict{F_\varphi}}$ and ${\corestrict{\corestrict{F_\varphi}}}$ to fill in the other vertical arrows. 
We define $\eta(\varphi)=\iota\circ {\corestrict{F_\varphi}}^{-1}$ 
to fill in the dotted arrow. Further, we already have a lift of the diagonal arrow $P\rightarrow\Pimage{\varphi}$, and we fill in $P\rightarrow P$ with the composite of isomorphisms ${\corestrict{\corestrict{F_\varphi}}}\circ{\corestrict{\varphi}}$: 
\begin{equation}   \label{diagram: factoring diagram in Lcal}
\begin{gathered}
\xymatrix{
P 
  \ar[r]^{\cong}
  \ar[ddr]_-{{\corestrict{\varphi}}}
  &P \ar[r]^-{\iota}
  & Q_{f_\varphi} 
     \ar@{-->}[dr]^-{\eta(\varphi)=\iota\circ {\corestrict{F_\varphi}}^{-1}}
     \ar[rr]^-{\iota}
  && N_{S}(P)
\\
&&&Q
\\
&\Pimage{\varphi}  
  \ar[r]_-{\iota}
  \ar[uu]_-{\corestrict{\corestrict{F_\varphi}}}^{\cong}
  &N_S(\Pimage{\varphi})\cap Q 
     \ar[ur]_-{\iota}
     \ar[rr]_-{\iota}
  \ar[uu]_-{{\corestrict{F_\varphi}}}^{\cong}
&& N_{S}(\Pimage{\varphi}).
  \ar[uu]_-{F_\varphi}
}
\end{gathered}
\end{equation}
Given $\varphi$ and choices of $f_\varphi$ and $F_\varphi$, the rest of the diagram is uniquely determined.  We define 
\begin{equation}\label{eq: define r}
r(\varphi)
=\iota_P^{Q_{f_{\varphi}}}\circ 
\left({\corestrict{\corestrict{F_\varphi}}}
     \circ {\corestrict{\varphi}}\right)
\in\Obj(\Normalizer{\Lcal})
\end{equation}
and we observe that 
\begin{equation}
    \eta(\varphi)=\iota\,_{N_Q(P_\varphi)}^{Q_{f_{\varphi}}}\circ {\corestrict{F_\varphi}}^{-1}
\end{equation}
is a morphism in $P\downto\Lcal$ from $r(\varphi)$ to $\varphi$. \\

\begin{remark}~ 
\begin{enumerate}
    \item If $\varphi\colon P\rightarrow Q$ is in $\Lcal$ and $\pi(\varphi)(P)=P$, then $Q_{f_\varphi}=N_Q P$.
\item $r(P\xrightarrow{\ \iota\ }Q)
     =(P\xrightarrow{\ \iota\ }N_Q P)$
\end{enumerate}
\end{remark}

\medskip

We extract some further details of from the construction of $r$ for later use. One result we need is that the retraction $r$ preserves the property of not being $\Fcal$-isomorphic to~$P$. 

\begin{lemma}     \label{lemma: proper containment new} 
If $\varphi\colon P\rightarrow Q$ is not an isomorphism, then $Q_{f_{\varphi}}$ properly contains~$P$. 
\end{lemma}

\begin{proof}
In the bottom row of diagram~\eqref{diagram: factoring diagram in Fcal new}, the middle entry is actually $N_Q(P_\varphi)$. If $\varphi$ is not an isomorphism, then $Q$ properly 
contains~$P_\varphi$, so $N_Q(P_\varphi)$ properly contains~$P_\varphi$ as well (\cite[Lemma~1.8]{BLO-Discrete}). It follows that in the top row, $Q_{f_\varphi}$ properly contains~$P$, as required. 
\end{proof}

We also need to know that $r$ actually is a retraction.

\begin{lemma}     \label{lemma: r properties}
If $\varphi\colon P\rightarrow Q$ is a morphism in $N_\Lcal P$, then $r(\varphi)=\varphi$, i.e., $r$ is a retraction.
\end{lemma}

\begin{proof}
By Lemma~\ref{lemma: P in Q}, we know that $P\subgroupeq Q\subgroupeq N_S P$ and $\pi(\varphi)(P)=P$. Hence we have $P_\varphi=P$ and $N_S(P_\varphi)\cap Q=Q$, since $Q\subgroupeq N_S P$. Because  $F_\varphi$ is the identity, we find $Q_{f_\varphi}=Q$ and 
$\iota_P^{Q_{f_{\varphi}}}\circ 
\left({\corestrict{\corestrict{F_\varphi}}}
     \circ {\corestrict{\varphi}}\right)
     =\iota_P^Q\circ {\corestrict{\varphi}}=\varphi$.
\end{proof}

Lastly, we must still establish that $r$ is actually a functor, despite the choices that were made during its construction. 

\begin{lemma}     \label{lemma: r is retraction functor}
~\begin{enumerate}
\item The retraction $r$ is a functor from $P\downto\Lcal$ to $\Normalizer{\Lcal}$. 
\item 
Let $i: (\Normalizer{\Lcal})\to (P\downarrow \Lcal)$  denote the inclusion functor. Then there is a natural transformation $\eta$ from $i\circ r$ to the identity on $P\downarrow \Lcal$.
\end{enumerate}
\end{lemma}

\begin{proof}
To prove that $r$ is a functor, we must establish that, despite having to make choices in defining $r(\varphi)$ for each object $\varphi$ of $P\downto\Lcal$, we can choose compatible morphisms in $\Normalizer{\Lcal}$ between the objects $P\rightarrow Q_{f_\varphi}$ for different~$\varphi$, making $r$ into a functor. 

Our strategy is to prove that a morphism $\beta$ in $P\downto\Lcal$ (which is a commuting diagram in~$\Lcal$ under~$P$ as on the left), together with choices of $f_\varphi$ and $f_{\varphi'}$, gives rise to a unique $\hat{\beta}$ in $N_\Lcal P$ making a commutative ladder in~$\Lcal$ (on the right):
\begin{equation}     \label{diagram: required diagram}
\begin{gathered}
\xymatrix{
P \ar[r]^{\varphi}
  \ar[d]_{=}
& Q\ar[d]_{\beta}\\
P \ar[r]_{\varphi'}
  &Q'
}
\end{gathered}
\qquad
\Rightarrow
\qquad
\begin{gathered}
\xymatrix{
P \ar[r]^{r(\varphi)}
  \ar[d]^-{=}
  & Q_{f_\varphi} 
  \ar[r]^-{\eta(\varphi)}
    \ar@{-->}[d]^-{\exists ! \hat{\beta}}
  & Q\ar[d]^-{\beta}\\
P  \ar[r]_{r(\varphi')}
  &Q'_{f_{\varphi'}}
  \ar[r]_{\eta(\varphi')}
& Q'. 
}
\end{gathered}
\end{equation}
While $\hat{\beta}$ depends not only on $\beta$ but also on the choices of $f_{\varphi}$ and $f_{\varphi'}$
(and their lifts to the linking system), we omit that dependence from the notation.

Our setup allows us to construct the commuting diagram 
\eqref{diagram: define betas new} in~$\Lcal$ below, working from right to left. (The notation is as in diagram~\eqref{diagram: factoring diagram in Lcal}.) 
Once again we denote the preferred ``inclusions" in $\Lcal$ simply by~$\iota$ to reduce clutter. 
We are given $\beta$ and the commutativity of the outermost rectangle by hypothesis. 
The composites $P\rightarrow Q_{f_{\varphi}}$ and $Q_{f_{\varphi}}\rightarrow Q$ across the top row are $r(\varphi)$ and $\eta(\varphi)$, respectively, and similarly are $r(\varphi')$ and $\eta(\varphi')$ for the second row. 
\begin{equation}     \label{diagram: define betas new}
\begin{gathered}
\xymatrix{
P \ar[r]_{{\corestrict{\varphi}}}
  \ar[d]_-{=}
     \ar@/_-1.3pc/[rrrrr]^-{\varphi} 
  & \Pimage{\varphi}
     \ar[r]_-{\corestrict{\corestrict{F_\varphi}}} 
  & P 
     \ar[r]_-{\iota}
  & Q_{f_{\varphi}}
     \ar[r]^-{({\corestrict{F_{\varphi}}})^{-1}}_-{\cong}
     \ar[d]^{\hat{\beta}}
  & N_S(P_\varphi)\cap Q 
      \ar[d]_{\corestrict{\beta}}
      \ar[r]_-{\iota}
  & Q
  \ar[d]^-{\beta}
\\
P  \ar[r]^{{\corestrict{\varphi'}}}
  \ar@/_1.4pc/[rrrrr]_-{\varphi'} 
  & \Pimage{\varphi'}
    \ar[r]^-{\corestrict{\corestrict{F_{\varphi'}}}}
  & P
     \ar[r]^-{\iota}
  & \ Q_{f_{\varphi'}}\ 
      \ar[r]_-{\ ({\corestrict{F_{\varphi'}}})^{-1}\ }^-{\cong}
  & N_S(P_{\varphi'})\cap Q'  
      \ar[r]^-{\iota}
  &Q'
}
\end{gathered}
\end{equation}
First, $\beta\circ\iota$ corestricts uniquely to $\corestrict{\beta}$ so that the rightmost square commutes. The map $\hat{\beta}$ is then uniquely defined by requiring commutativity of the next square to the left. 

To test if the left rectangle commutes, we use the result that every morphism in $\Lcal$ is a categorical monomorphism (\cite[Prop.~A.2(d), Cor.~A.5]{BLO-LoopSpaces}). It is sufficient to check that composing both ways around the left square with the composite 
$\iota\circ {\corestrict{F_{\varphi'}}}
     \colon Q_{f_{\varphi'}}\rightarrow Q'$
are the same. The two compositions are the same by commutativity of the other two squares and the outer rectangle. 

Lastly, uniqueness of $\hat{\beta}$ guarantees functoriality of~$r$ despite the choices made in the construction. 

For (2), observe that \eqref{diagram: required diagram}, thought of as a diagram in $P\downarrow \Lcal$, is exactly the diagram required to show that $\eta$ is a natural transformation from $i\circ r$ to the identity.
\end{proof}

We can now put together the proof of Proposition~\ref{proposition: deformation retraction new}, whose statement we reproduce for convenience.

\begin{RetractionProposition}
\RetractionPropositionText
\end{RetractionProposition}

\begin{proof}
The function 
$r\colon (P\downto\Lcal)\rightarrow (\Normalizer{\Lcal})$ of Construction~\ref{construction: retraction to normalizer new} is a retraction functor 
by Lemmas \ref{lemma: r properties} and~\ref{lemma: r is retraction functor}, and it preserves non-isomorphisms by Lemma~\ref{lemma: proper containment new}. The natural transformation $\eta$ from Lemma~\ref{lemma: r is retraction functor} restricts to a natural transformation from $i\circ r$ to the identity on $P\downtoNoniso\Lcal$.
Hence $r$ induces a homotopy equivalence between nerves. 
\end{proof}

\section{The subgroup $\extension$}
\label{section: Phat}

The overall strategy for proving Theorem~\ref{theorem: reduce linking to centric radical} is to study
a sequence of \undercategories\ in order to apply Quillen's Theorem~A. 
In Section~\ref{section: normalizer fusion subsystems}, we established a homotopy equivalence between the nerves of $P\downtoNoniso\Lcal$ and $\NormalizerPrime{\Lcal}$. 
In this section, we construct a supergroup $\extension$ of~$P$ inside of~$N_SP$, and a functor from $\NormalizerPrime{\Lcal}$ to $\extension\downto N_\Lcal P$. Our goal is the following proposition. 

\begin{proposition}     \label{proposition: inflate to Phat}
If $P$ is fully normalized, $\Fcal$-centric, and not $\Fcal$-radical, then 
$\realize{\NormalizerPrime{\Lcal}}
\simeq
\realize{\extension\downto N_\Lcal P}$. 
\end{proposition}

The first part of the section leads up to the definition of $\extension$ and its elementary properties (equation~\eqref{eq: definition of Phat} and Lemma~\ref{lemma: prop of Pwiggle}). The second part of the section sets up and proves the key extension property of~$\extension$ (Lemma~\ref{lemma: extension to Phat}), and the section concludes with the proof of Proposition~\ref{proposition: inflate to Phat}. 
We begin with an observation about conjugation by images of an extension of a morphism. 

\begin{lemma}      \label{lemma: conjugation}
Let $\alpha\in\Aut_\Fcal(P)$, and assume that $\alpha$ extends to $\alphawiggle\colon Q\rightarrow S$ in~$\Fcal$, 
where $P\subgroupeq Q\subgroupeq N_S P$ and $\restr{\alphawiggle}{P}=\alpha$. If $y\in Q$, then as automorphisms of~$P$,
\[
\alpha\circ c_y\circ \alpha^{-1} = c_{\alphawiggle(y)}.
\]
\end{lemma}

\begin{proof}
Remembering that $\alpha$ itself cannot be applied to~$y$, we compute
\begin{align*}
\alpha\circ c_y\circ \alpha^{-1}(p)
   &= \alpha\left(y\cdot (\alpha^{-1}p)\cdot y^{-1}\right)\\
   &= \alphawiggle\left(y\cdot (\alpha^{-1}p)\cdot y^{-1}\right)\\
   &= \alphawiggle(y) \cdot p \cdot \alphawiggle(y^{-1}). \qedhere
\end{align*}
\end{proof}

Motivated by Lemma~\ref{lemma: conjugation}, one makes the following definition. 

\begin{definition}   \label{definition: Nalpha}
Let $\alpha\in\Aut_{\Fcal}(P)$. We define 
\[
N_\alpha
\definedas \{
y\in N_S P \suchthat 
\alpha\circ c_y\circ \alpha^{-1}\in\Aut_S(P)
\}.
\]
\end{definition}

As a corollary of Lemma~\ref{lemma: conjugation} and Definition~\ref{definition: Nalpha}, we find that $N_\alpha$ is the largest subgroup of $N_S P$ over which $\alpha$ could possibly extend. 

\begin{corollary}   \label{corollary: contained in N}
If $\alpha\in\Aut_{\Fcal}(P)$ extends 
to $\alphawiggle\colon Q\rightarrow S$ where 
$P\subgroupeq Q\subgroupeq N_S P$, then $Q\subgroupeq N_\alpha$. 
\end{corollary}

There is also a uniqueness property for extensions of elements of $\Aut_\Fcal(P)$, as described by the following lemma. The lemma considerably strengthens what one could conclude just from Lemma~\ref{lemma: conjugation}. 

\begin{lemma}\cite[Prop.~2.8]{BLO-Discrete}
\label{lemma: uniqueness of extension}
Let $P$ be $\Fcal$-centric, and suppose $P\subgroupeq Q$. If
$\alpha\in\Aut_\Fcal(P)$, and $\alphawiggle$ and $\alphawiggle'$ are extensions of $\alpha$ to morphisms~$Q\rightarrow S$, then there exists $z\in Z(P)\subgroupeq Q$ such that $\alphawiggle'=\alphawiggle\circ c_z$. In particular, $\alphawiggle(Q)=\alphawiggle'(Q)$.
\end{lemma} 



We are interested in the largest supergroup of $P$ over which \textit{all} $\Fcal$-automorphisms of $P$ must extend. 
If $P$ is a fully normalized subgroup, let $\extension$ be defined by
\begin{equation}
\label{eq: definition of Phat}
\extension\definedas\bigcap_{\alpha\in\Aut_\Fcal(P)}N_\alpha. 
\end{equation}

\begin{lemma}   \label{lemma: prop of Pwiggle}
If $P$ is fully normalized, then $P\triangleleft\extension\triangleleft  N_S P$, and 
$C_S(P)\subgroupeq\extension$. 
\end{lemma}

\begin{proof}
The centralizer of $P$ is contained in every $N_\alpha$, and therefore in~$\extension$. 

Let $x\in N_S P$. We claim that 
\[
x\, N_{\alpha}\, x^{-1}= N_{\alpha\circ c_{x^{-1}}}.
\]
To see the inclusion from left to right, suppose that $\alphawiggle\colon N_\alpha\rightarrow S$ is an extension of~$\alpha$. Then 
$\alphawiggle\circ c_{x^{-1}}\colon x\, N_{\alpha}\, x^{-1}\rightarrow S$ is an extension of 
$\alphawiggle\circ c_{x^{-1}}\in\Aut_\Fcal(P)$, and so
$x\, N_{\alpha}\, x^{-1}\subgroupeq N_{\alpha\circ c_{x^{-1}}}$ 
by Corollary~\ref{corollary: contained in N}. The reverse inclusion is the same argument. 

The preceding paragraph proves that conjugation by elements of $N_S P$ permutes the groups $N_\alpha$ for various $\alpha\in\Aut_S(P)$ and therefore stabilizes their intersection, namely~$\extension$.
\end{proof}

With the basic properties of $\extension$ in place, we consider its extension properties. The goal is Lemma~\ref{lemma: extension to Phat}, which establishes the existence and uniqueness 
of certain extensions of automorphisms of~$P$ over subgroups of $N_S P$ containing~$\extension$.  

\begin{lemma}   \label{lemma: image of extension}
Let $P$ be fully normalized, let $\alpha\in\Aut_{\Fcal}(P)$, and let $\alphawiggle\colon N_\alpha\rightarrow S$ be an extension of~$\alpha$. Then
$\im\big(\extension\xrightarrow{\alphawiggle} S\big)=\extension.$
\end{lemma}

\begin{proof}
It suffices to show that $\alphawiggle(\extension) \subgroupeq \extension$ because we then have a subgroup containment between groups of equal size (in
the sense of Definition~\ref{def:dpt}), which are therefore equal.
Let $y\in\extension$.
First we check that $\alphawiggle(y)\in N_S P$.
By Lemma~\ref{lemma: conjugation}, we know that 
\[
c_{\alphawiggle(y)}=\alpha\circ c_y\circ \alpha^{-1},
\]
which stabilizes~$P$ because $y\in N_S P$, 
and therefore $\alphawiggle(y)$ is in~$N_S P$. 

To finish the proof, we must show that $\alphawiggle(y)\in N_{\beta}$ for all $\beta\in\Aut_{\Fcal}(P)$. Applying Lemma~\ref{lemma: conjugation} again, we find
\begin{align*}
\beta\circ c_{\alphawiggle(y)}\circ \beta^{-1}
   &= \beta\circ (\alpha\circ c_y\circ \alpha^{-1})\circ \beta^{-1}\\
   &= (\beta\circ\alpha)\circ c_y\circ (\beta\circ\alpha)^{-1}.
\end{align*}
Since $y\in\extension$, we know $y \in N_{\beta\circ\alpha}$, so the last automorphism is in~$\Aut_S(P)$ as required. 
\end{proof}

Recall that for a finite group~$H$, we write $O_p(H)$ for the largest normal \pdash subgroup of~$H$, i.e. the intersection of all Sylow \pdash subgroups of~$H$. Only the first half of the proof of the following proposition is necessary for Corollary~\ref{corollary: radically proper}, but the equality statement shows that our $\extension$ agrees with the group $\Phat$ that plays a similar role in the proof of \cite[Prop.~3.11]{BCGLO}. 


\begin{proposition}\label{proposition: Pwiggle contains O_p}
Let $P$ be fully normalized. The image of $\extension$ under the natural map $c\colon N_S P\rightarrow\Out_\Fcal(P)$ equals $O_p(\Out_\Fcal(P))$. 
\end{proposition}

\begin{proof}
First we prove that 
$O_p(\Out_\Fcal(P))\subgroupeq c(\extension)$. 
Because $P$ is fully normalized, we know that $O_p(\Out_\Fcal(P))\subgroupeq\Out_S(P)$, so any element of 
$O_p(\Out_\Fcal(P))$ can be represented by $c_w$ for some $w\in N_S P$. We would like to show that for any $\beta\in\Aut_\Fcal(P)$, we have $w\in N_\beta$, so as to conclude that $w\in\extension$. 

Because $O_p(\Out_\Fcal(P))$ is contained in $\Out_S(P)$ and is normal in $\Out_\Fcal(P)$, for any $[\beta]\in\Out_\Fcal(P)$, there exists $s\in N_S P$ such that 
\[
[\beta]\cdot [c_w]\cdot [\beta^{-1}]=[c_s],
\]
and by adjusting the choice of $s$ using an element of~$P$ if necessary, we can assume that $\beta \cdot c_w\cdot \beta^{-1}=c_{s}$. Therefore $w\in N_\beta$. Since $\beta$ was arbitrary, we find that $w\in\extension$, as required. 

Next we show that $c(\extension)\subgroupeq O_p(\Out_\Fcal(P))$. It suffices to
show that $c(\extension)$ is a normal subgroup of $\Out_\Fcal(P)$, and since
the quotient map $\Aut_\Fcal(P)\to \Out_\Fcal(P)$ is surjective, it suffices to
show that the image of $\extension$ in $\Aut_\Fcal(P)$ is normal in $\Aut_\Fcal(P)$.
Let $y\in\extension$ and $\beta\in\Aut_\Fcal(P)$; we must show that $\beta \circ c_y
\circ \beta^{-1} = c_z$ for some $z\in \extension$.
By Lemma~\ref{lemma: conjugation} we know that $\beta\circ
c_y\circ\beta^{-1}=c_{\betawiggle(y)}$ as automorphisms of~$P$, where
$\betawiggle$ is the extension of $\beta$ to $N_\beta$ guaranteed by Axiom (II)
of saturation. Moreover,
$\betawiggle(y)$ is in $\extension$ by Lemma~\ref{lemma: image of extension}.
\end{proof}

The following corollary is the first of two critical ingredients in the proof of Proposition~\ref{proposition: inflate to Phat}, the other being the extension property proved in Lemma~\ref{lemma: extension to Phat}. 

\begin{corollary}    \label{corollary: radically proper}
Let $P$ be a fully normalized subgroup. If $P$ is not $\Fcal$-radical, then $\extension$ properly contains~$P$.
\end{corollary}

\begin{proof}
If $P$ is not $\Fcal$-radical, then by definition $O_p(\Out_\Fcal(P))$ is nontrivial.
By Proposition~\ref{proposition: Pwiggle contains O_p}, the inverse image of 
$O_p(\Out_\Fcal(P))$ along 
$N_S P\to
\Out_\Fcal(P)$ is~$\extension$, which therefore properly contains~$P$. 
\end{proof}

Finally, we establish that for $\Fcal$-centric subgroups, we can extend maps in the $\Lcal$-normalizer of~$P$. 

\begin{lemma}
\label{lemma: extension to Phat} 
Let $P$ be $\Fcal$-centric, and let $Q,Q'\subgroupeq N_S P$ with $P\subgroupeq Q\cap Q'$. 
Given $\varphi\in \Hom_{N_\Lcal(P)}(Q,Q')$,
there exists a unique
  $\hat{\varphi}
  \in 
  \Hom_{N_\Lcal P}(\extension\cdot Q, \extension \cdot Q')$
such that $\restr{\hat{\varphi}}{Q}=\varphi$.
\end{lemma}

\begin{proof}

Observe first that $Q,Q'\subgroupeq N_S P$, so they both
normalize~$\extension$ by Lemma~\ref{lemma: prop of Pwiggle}; thus $\extension\cdot Q$ and $\extension\cdot Q'$ are groups, contained in $N_S P$. Let $\alpha \in \Aut_{\F}(P)$ 
be the restriction to $P$ of 
$\pi(\varphi)\in \Hom_{N_\Fcal P}(Q,Q')$.
Then $Q\subgroupeq N_\alpha$ and also $\extension \subgroupeq N_\alpha$ 
by Corollary~\ref{corollary: contained in N}. Hence 
$\extension\cdot Q\subgroupeq N_\alpha$, and 
by axiom (II) of saturation, there exists $f_\alpha\colon \extension \cdot Q\rightarrow S$ in $\Fcal$ with $\restr{f_\alpha}{P}=\alpha$. 
Further, by Lemma~\ref{lemma: image of extension} we know that $f_\alpha(\extension)\subgroupeq\extension$.

By definition, $\pi(\varphi):Q\rightarrow Q'$, like~$f_\alpha$, is an extension of $\alpha$. By Lemma~\ref{lemma: uniqueness of extension}, there exists
$z\in Z(P)$ such that $\restr{f_\alpha}{Q}=\pi(\varphi)\circ c_z$, so by replacing $f_\alpha$ with $f_\alpha\circ c_{z^{-1}}$ if necessary, we can assume that $\restr{f_\alpha}{Q}=\pi(\varphi)$ and still satisfies $f_\alpha(\extension)=\extension$. Hence $f_\alpha\in\Hom_{N_\Fcal P}(\extension\cdot Q, \extension\cdot Q')$. 

Choose a lift $\widetilde{\varphi}$ of $f_\alpha$ to $N_\Lcal(P)$.
We now have the commutative diagram below in~$N_\Fcal P$ (on the left) and a proposed lifting of that diagram to $N_\Lcal P$ on the right
that may or may not commute:
\[
\xymatrix{
\extension\cdot Q \ar[r]^{f_\alpha} &\extension\cdot Q'\\
Q\ar[r]_{\pi(\varphi)}\ar[u]^-{\subgroupeq}
& Q'\ar[u]_-{\subgroupeq}
}
\qquad\qquad
\xymatrix{
\extension\cdot Q \ar[r]^{\widetilde{\varphi}} &\extension\cdot Q'\\
Q\ar[r]_{\varphi}\ar[u]^-{\iota} 
& Q'.\ar[u]_-{\iota}
}
\]
Both compositions around the right-hand square project to the same map in~$\Fcal$, because the left-hand square commutes in~$\Fcal$. Since $P\subgroupeq Q$, we know that $Q$ is $\Fcal$-centric, so by the axioms of a linking system there exists $x\in C_S(Q)=Z(Q)$ such that 
\begin{align*}
\iota_{Q'}^{\extension\cdot Q'}\circ\varphi
&=
\widetilde{\varphi}\circ\iota_Q^{\extension\cdot Q}\circ \delta_Q(x)\\
& =\left(\widetilde{\varphi}\circ \delta_{\extension\cdot Q}(x)\right)\circ \iota_Q^{\extension\cdot Q}
\end{align*}
where the second line uses property (C) in the definition of a linking system.
Since $\hat{\varphi}:=\widetilde{\varphi}\circ \delta_{\extension\cdot Q}(x)$ restricts to $\varphi$ and
$\varphi$ preserves $P$, so does $\hat{\varphi}$, and hence $\hat{\varphi}$ is
in $N_\Lcal P$.
Lastly, $\hat{\varphi}$ is unique because $\iota_{Q}^{\extension\cdot Q}$ is an epimorphism in a categorical sense (\cite[Prop.~A.2]{BLO-LoopSpaces}).
\end{proof}

The proof of Proposition~\ref{proposition: inflate to Phat} is now a routine matter of checking diagrams. 

\begin{proof}[Proof of Proposition~\ref{proposition: inflate to Phat}]
We exhibit functors in both directions between 
$\NormalizerPrime{\Lcal}$ and $\extension\downto N_\Lcal P$, with appropriate natural transformations.
We define a functor 
$G:(\extension \downarrow N_\Lcal P)
\to (\NormalizerPrime{\Lcal})$ 
by precomposing with the distinguished inclusion $P\xrightarrow{\iota}\extension$, which is a morphism of $N_\Lcal P$. The image of $G$ is $\NormalizerPrime{\Lcal}$ by Corollary~\ref{corollary: radically proper}.

In the other direction, to define
$F:(\NormalizerPrime{\Lcal})\to (\extension\downto N_\Lcal P)$, suppose that 
$P\xlongrightarrow{\varphi}Q$ is an object of $\NormalizerPrime{\Lcal}$. Then 
$P$ is a subgroup of~$Q$, so by Lemma~\ref{lemma: extension to Phat} there exists a unique morphism
$\hat{\varphi}\colon\extension\rightarrow\extension\cdot Q$
of $\Normalizer{\Lcal}$ that extends~$\varphi$, and we define $F(\varphi)=\hat{\varphi}$. We define $F$ on a morphism $\beta$ by
\begin{equation}    \label{eq: F on morphism}
F\left(
\begin{gathered}
\xymatrix{\vspace{-100pt}
P\ar[r]^\varphi\ar[rd]_-{\varphi'} & Q\ar[d]^-\beta
\\ & Q'
}
\end{gathered}\right)
=
\left(\begin{gathered}\xymatrix{
\extension\ar[r]^{\hat{\varphi}}\ar[rd]_-{\hat{\varphi'}} & \extension \cdot Q\ar[d]^-{\hat{\beta}}
\\ 
& \extension \cdot Q'
}\end{gathered}
\right)
\end{equation}
where $\hat{\beta}$ is the (unique) extension of $\beta$ guaranteed by Lemma~\ref{lemma: extension to Phat}. 

We want a natural transformation $\nu$ from the identity functor on $\NormalizerPrime{\Lcal}$
to the composite $GF$. To define $\nu(P\xrightarrow{\varphi}Q)$, we need a morphism 
$Q\rightarrow \extension\cdot Q$ of $\Lcal$ under~$P$, and we use 
\[
\xymatrix{
P \ar[r]^\varphi
  \ar[rd]_-{GF(\varphi)=\hat{\varphi}\circ\iota_{P}^{\extension}} 
   & Q\ar[d]^{\nu(\varphi)\definedas\iota_Q^{\extension\cdot Q}}
\\
& \extension \cdot Q.
}
\]
where the diagram commutes because $\restr{\hat{\varphi}}{P}=\varphi$.

To check naturality of $\nu$ for the morphism $\beta$ shown in~\eqref{eq: F on morphism}, consider the diagram below, where the front triangle is $\nu(\varphi):\varphi\to GF(\varphi)$, the back triangle is $\nu(\varphi'):\varphi'\to GF(\varphi')$, and $\hat{\beta} = GF(\beta)$.
The unlabelled map is 
$GF(\beta\circ \varphi) 
   = \widehat{\beta\circ \varphi}\circ\iota_P^{\extension} 
   = \iota_{Q'}^{\extension\cdot Q'}\circ \beta\circ \varphi$.
\begin{equation}   \label{diagram: prism}
\begin{gathered}
\xymatrix@R=30pt@C=30pt{
& P\ar[r]^-{\varphi'}\ar[rd]|!{[d];[r]}\hole 
& Q'\ar[d]^-{\iota_{Q'}^{\extension \cdot Q'}}
\\
P \ar@{=}[ru]
  \ar[r]^-\varphi\ar[rd]_-{\hat{\varphi}\circ \iota_P^{\extension}} 
&
Q \ar[d]^-{\iota_Q^{\extension\cdot Q}}
  \ar[ru]|<<<<<{\,\beta\,} 
& 
\extension\cdot Q'
\\
& \extension \cdot Q
   \ar[ru]_*-<1ex>{\labelstyle \ensuremath{\hat{\beta}}}
}
\end{gathered}
\end{equation}
The rectangle on the right commutes by the construction of $\hat{\beta}$ (Lemma~\ref{lemma: extension to Phat}). The
back rectangle commutes because the other faces
commute, establishing naturality of~$\nu$. 

The natural transformation 
$\Id\Rightarrow FG:(\extension\downto N_\Lcal P) \to (\extension \downto N_\Lcal P)$ applied to $\varphi:\extension\to Q$ (where $Q\subseteq N_SP$) is the morphism
\[
\xymatrix{
\extension\ar[r]^-\varphi\ar[rd]_-{\widehat{\varphi\circ\iota_P^{\extension}}} & Q\ar[d]^{\iota_Q^{\extension\cdot Q}}
\\ & \extension\cdot Q.
}
\]
The diagram commutes by the uniqueness in Lemma~\ref{lemma: extension to Phat}, since 
both ways around the diagram are extensions of the composite $P\too{\iota_P^{\extension}}\extension \too{\varphi}Q$ to a map $\extension\rightarrow\extension\cdot Q$. 
Naturality follows from a diagram similar to \eqref{diagram: prism} with $P$ replaced by~$\extension$ and the slanted maps adjusted accordingly. 
\end{proof}

\section{Proof of Theorem~\ref{theorem: reduce linking to centric radical}}
\label{section: proof of theorem}

In this section, we prove the main theorem of the paper, which we reproduce for the reader's convenience. 

\begin{maintheorem}
\maintheoremtext
\end{maintheorem}

First we note that we can work within the subcategory $\Lcal^\bullet$ of~$\Lcal$ (see Proposition~\ref{proposition: bullet}). 

\begin{lemma}     \label{lemma: reduce} 
With the notation of Theorem~\ref{theorem: reduce linking to centric radical}, 
$\Lcal^\Hcal\hookrightarrow\Lcal$ induces a homotopy equivalence of nerves 
if and only if 
$\Lcal^{\Hcal^\bullet}\hookrightarrow\Lcal^\bullet$ induces a homotopy equivalance of nerves. 
\end{lemma}

\begin{proof}
Because $\Hcal^\bullet\subseteq\Hcal$, there is a commuting diagram 
\[
\xymatrix{
\Lcal^{\Hcal^\bullet} \ar[r]\ar[d] &\Lcal^\bullet\ar[d] \\
\Lcal^\Hcal 
   \ar[r] 
   \ar@/_-.7pc/[u]
   & \Lcal
   \ar@/_.7pc/[u],
}
\]
where the downward vertical arrows are given by inclusion of subcategories and the upward arrows are given by the functor 
$P\mapsto P^\bullet$. The vertical arrows induce homotopy equivalences of categories because the down-and-up composite is the identity on the top row, and the distinguished inclusions $\iota_P^{P^{\bullet}}$ provide a natural transformation from the identity functor on the bottom row to the up-and-down composite 
\cite[Prop.~4.5~(a)]{BLO-Discrete}.
\end{proof}

The proof of Theorem~\ref{theorem: reduce linking to centric radical}
follows the general argument of~\cite[Thm.~3.5]{BCGLO}, while dealing with the changes needed for the infinite situation. By Lemma~\ref{lemma: reduce}, it is sufficient to assume that $\Hcal=\Hcal^\bullet$, i.e. that $\Lcal^\Hcal$ is a (full) subcategory of~$\Lcal^\bullet$. 
Hence we can start with $\Lcal^\bullet$, and get to $\Lcal^\Hcal$ by inductively pruning conjugacy classes of subgroups that are not in~$\Hcal$, starting with the smallest subgroups. Since $\Lcal^\bullet$ has a finite number of conjugacy classes, this process terminates in~$\Lcal^{\Hcal}$.

\begin{proof}[Proof of Theorem~\ref{theorem: reduce linking to centric radical}]
By Lemma~\ref{lemma: reduce}, it is sufficient to assume that $\Lcal^\Hcal\subseteq\Lcal^\bullet$. Since $\Lcal^\bullet$ has a finite number of conjugacy classes of objects, we can make a finite list of conjugacy classes $\langle P_1\rangle,\ldots,\langle P_n\rangle$ of $\Lcal^\bullet$ such that 
\begin{itemize}
    \item $P_1,\ldots,P_n$ are fully normalized representatives of distinct conjugacy classes of objects of~$\Lcal^{\bullet}$, and represent all conjugacy classes that are in $\Lcal^\bullet$ but not in~$\Hcal$. 
    \item The list of sizes is non-increasing: $\size{P_i}\geq \size{P_{i+1}}$ for $i=1,..,n-1$. 
\end{itemize}
Let $\Lcal_0=\Lcal^\Hcal$, and for each $i=0,\ldots,n-1$, let 
$\Lcal_{i+1}$ be the full subcategory of $\Lcal^\bullet$ whose objects are $\Obj(\Lcal_{i})\cup\langle P_{i+1}\rangle $. Thus 
\begin{equation}     \label{eq: sequence of categories}
\Lcal^\Hcal=\Lcal_0\subsetneq\Lcal_1\subsetneq\ldots\subsetneq\Lcal_n=\Lcal^\bullet,
\end{equation}
and each subcategory in the sequence contains one more isomorphism class of $\Obj(\Lcal)$ than the one before it. 

We seek to prove that for all~$i$, the inclusion 
$\Lcal_i\subsetneq\Lcal_{i+1}$ induces a homotopy equivalence of nerves, and
we wish to apply Quillen's Theorem~A. We must show that for all $Q\in\Obj(\Lcal_{i+1})$, the nerve of $Q\downto\Lcal_{i}$ is contractible. If $Q$ is actually an object of $\Lcal_{i}$, then the desired statement is true because the identity map of $Q$ is an initial object. Thus, we need only consider $Q\in\langle P_{i+1}\rangle$. For all such~$Q$, the \undercategories\ 
$Q\downto\Lcal_i$ are isomorphic, 
so we need only prove that the nerve of $P_{i+1}\downto \Lcal_i$ is contractible (where we have assumed that $P_{i+1}$ is fully normalized). 
To simplify notation, let $P\definedas P_{i+1}$, which remains fixed for the remainder of the proof.

First we assert that we have an isomorphism of categories
\[
(P\downto\Lcal_{i}) \cong (P\downtoNoniso\Lcal^\bullet).
\]
To see this, note that if $P\rightarrow R$ is a morphism of~$\Lcal^\bullet$ that is not an isomorphism, then $R\in\Obj(\Lcal_i)$, because $\size{R}>\size{P}$. 
Hence it is sufficient to prove that for all fully normalized subgroups~$P$, the nerve of $P\downtoNoniso\Lcal^\bullet$ is contractible. 

Consider the sequence of categories 
\[
\xymatrix{
P\downtoNoniso\Lcal^{\bullet}
      \ar@/_1pc/[r]_{\subseteq}
   & P\downtoNoniso\Lcal
      \ar@/_1pc/[r]_{r}
      \ar@/_1.2pc/[l]_{(\whatever)^\bullet}
   & \NormalizerPrime{\Lcal}
      \ar@/_1.2pc/[l]_{\supseteq}
      \ar@/_1pc/[r]_{\varphi\mapsto\hat{\varphi}}
   & \extension\downto N_\Lcal P.
      \ar@/_1.2pc/[l]_{(\whatever)\circ\iota_P^{\extension}}
}
\]
We assert that each adjacent pair has homotopy equivalent nerves. For the first pair, the map from right to left takes a non-isomorphism $P\rightarrow Q$ to the composite $P\rightarrow Q\rightarrow Q^\bullet$, 
which is likewise a non-isomorphism, since $\size{Q^\bullet}\geq \size{Q}>\size{P}$. Hence the composite $(\whatever)^\bullet\circ(\subseteq)$ is the identity on $P\downtoNoniso\Lcal^{\bullet}$, and the natural map $Q\rightarrow Q^\bullet$ is a natural transformation from the identity on $P\downtoNoniso\Lcal$ to the composite $(\subseteq)\circ (\whatever)^\bullet$. 

The middle pair has homotopy equivalent nerves by
Proposition~\ref{proposition: deformation retraction new}.
Lastly, because $P$ is fully normalized, $\Fcal$-centric, and not $\Fcal$-radical (because $P$ is not in~$\Hcal$), the third pair also has homotopy equivalent nerves
by Proposition~\ref{proposition: inflate to Phat}. 
However, the identity map $\extension\rightarrow\extension$ is an initial object of $\extension\downto N_\Lcal P$, which therefore has contractible nerve, so we conclude that
\[
\realize{P\downto\Lcal_i}
\cong \realize{P\downtoNoniso\Lcal^\bullet}
\simeq \realize{\extension\downto N_\Lcal P}
\simeq *.
\]

We have proved that in the sequence~\eqref{eq: sequence of categories}, each successive pair of categories has homotopy equivalent nerves by Quillen's Theorem~A. Since the inclusion $\Lcal^\bullet\subseteq\Lcal$ induces a homotopy equivalence of nerves \cite[Prop.~4.5]{BLO-Discrete}, the proof is finished. 
\end{proof}

\bibliographystyle{amsalpha}
\newcommand{\etalchar}[1]{$^{#1}$}
\providecommand{\bysame}{\leavevmode\hbox to3em{\hrulefill}\thinspace}
\providecommand{\MR}{\relax\ifhmode\unskip\space\fi MR }
\providecommand{\MRhref}[2]{%
  \href{http://www.ams.org/mathscinet-getitem?mr=#1}{#2}
}
\providecommand{\href}[2]{#2}

\end{document}